\documentclass[12pt]{amsart}
\usepackage{amsfonts, amsbsy, amsmath, amssymb}

\newtheorem{thm}{Theorem}[section]

\newtheorem{prop}[thm]{Proposition}

\newtheorem{rmk}[thm]{Remark}

\newtheorem{thm-con}[thm]{Theorem-Conjecture}
\numberwithin{equation}{section}

\theoremstyle{definition}

\newcommand{\f}{\Bbb F}

\begin{document}

\title[Fibonacci self-reciprocal polynomials and Fibonacci PPs]{Fibonacci self-reciprocal polynomials and Fibonacci permutation polynomials}

\author[Neranga Fernando]{Neranga Fernando}
\address{Department of Mathematics,
Northeastern University, Boston, MA 02115}
\email{w.fernando@northeastern.edu}

\author[Mohammad Rashid]{Mohammad Rashid}
\address{Northeastern University, Boston, MA 02115}
\email{rashid.m@husky.neu.edu}

\begin{abstract}
Let $p$ be a prime. In this paper, we give a complete classification of self-reciprocal polynomials arising from Fibonacci polynomials over $\mathbb{Z}$ and $\mathbb{Z}_p$, where $p=2$ and $p>5$. We also present some partial results when $p=3, 5$. We also compute the first and second moments of Fibonacci polynomials $f_{n}(x)$ over finite fields, which give necessary conditions for Fibonacci polynomials to be permutation polynomials over finite fields. 
\end{abstract}

\keywords{Fibonacci polynomial, permutation polynomial, self-reciprocal polynomial, finite field, Dickson polynomial}

\subjclass[2010]{11B39, 11T06, 11T55}

\maketitle

%%%%%%%%%%%%%%%%%%%%%%%%%%%%%%%%%%%%%%%%
%  section 1 
%%%%%%%%%%%%%%%%%%%%%%%%%%%%%%%%%%%%%%%%

\section{Introduction}

{\it Fibonacci polynomials} were first studied in 1833 by Eugene Charles Catalan. Since then, Fibonacci polynomials have been extensively studied by many for their general and arithmetic properties; see \cite{Abd-Elhameed-Youssri-El-Sissi-Sadek, Cigler, Hoggat-Bicknell, Koshy-2001, Levy, Tasyurdu-Deveci, Yuan-Zhang}. In a recent paper, Koroglu, Ozbek and Siap studied cyclic codes that have generators as Fibonacci polynomials over finite fields; see \cite{Koroglu-Ozbek-Siap}. In another recent paper, Kitayama and Shiomi studied the irreducibility of Fibonacci polynomials over finite fields; see \cite{Kitayama-Shiomi-2017}. 

Fibonacci polynomials are defined by the recurrence relation $f_0(x)=0$, $f_1(x)=1$, and
$$f_n(x)=xf_{n-1}(x)+f_{n-2}(x),\,\, \textnormal{for} \,\,n\geq 2.$$

The Fibonacci polyomial sequence is a generalization of the Fibonacci number sequence: $f_n(1)=F_n$ for all $n$, where $F_n$ denotes the $n$-th Fibonacci number. Moreover, $f_n(2)$ defines the well-known {\it Pell numbers} $1, 2, 5, 12, 19, \ldots$. Fibonacci polynomials can also be extended to negative subscripts (see \cite[Chapter 37]{Koshy-2001}): 
$$f_{-n}(x)=(-1)^{n+1}f_{n}(x).$$

In the first part of this paper, we explore self-reciprocal polynomials arising from Fibonacci polyonomials. The reciprocal $f^*(x)$ of a polynomial $f(x)$ of degree $n$ is defined by $f^*(x)=x^n\,f(\frac{1}{x})$. A polynomial $f(x)$ is called {\it self-reciprocal} if $f^*(x)=f(x)$, i.e. if $f(x)=a_0+a_1x+a_2x^2+\cdots +a_nx^n$, $a_n\neq 0$, is self-reciprocal, then $a_i=a_{n-i}$ for $0\leq i\leq n$. The coefficients of a self-reciprocal polynomial form a palindrome because of which a self-reciprocal polynomial is also called \textit{palindromic}. Many authors have studied self-reciprocal polynomials for their applications in the theory of error correcting codes, DNA computing, and in the area of quantum error-correcting codes. We explain one application in the next paragraph. 

Let $C$ be a code of length $n$ over $R$, where $R$ is either a ring or a field. The reverse of the codeword $c=(c_0,c_1,\ldots ,c_{n-2}, c_{n-1})$ in $C$ is denoted by $c^r$, and it is given by $c^r=(c_{n-1},c_{n-2},\ldots ,c_1, c_0)$. If $c^r\in C$ for all $c\in C$, then the code $C$ is defined to be reversible. Let $\tau$ denote the cyclic shift. Then $\tau(c)=(c_{n-1},c_0,\ldots ,c_{n-2})$. If the cyclic shift of each codeword is also a codeword, then the code $C$ is said to be a {\it cyclic code} . It is a well-known fact that cyclic codes have a representation in terms of polynomials. For instance, the codeword $c=(c_0,c_1,\ldots ,c_{n-1})$ can be represented by the polynomial $h(x)=c_0+c_1x+\cdots c_{n-1}x^{n-1}$. The cyclic shifts of $c$ correspond to the polynomials $x^ih(x)\pmod{x^n-1}$ for $i=0,1, \ldots , n-1$. There is a unique codeword among all non-zero codewords in a cyclic code C whose corresponding polynomial $g(x)$ has minimum degree and divides $x^n-1$. The polynomial $g(x)$ is called the generator polynomial of the cyclic code $C$. In \cite{Massey-1964}, Massey studied reversible codes over finite fields and showed that the cyclic code generated by the monic polynomial $g(x)$ is reversible if and only if $g(x)$ is self-reciprocal. 

We present a complete classification of self-reciprocal polynomials arising from Fibonacci polynomials over $\mathbb{Z}$ and $\mathbb{Z}_p$, where $p=2$ and $p>5$. According to our numerical results obtained from the computer, the cases $p=3$ and $p=5$ seem to be inconclusive since the number of patterns of $n$ is increasing as $n$ increases. However, we present some sufficient conditions on $n$ for $f_n$ to be self-reciprocal when $p=3$ and $p=5$. We refer the reader to \cite{Adleman-1994}, \cite{Guenda-Jitman-AG}, \cite{GG-2013} and \cite{Massey-1964} for more details about self-reciprocal polynomials. 

In the second part of the paper, we explore necessary conditions for Fibonacci polynomials to be permutation polynomials over finite fields. Permutation polynomials over finite fields also have many important applications in coding theory. Let $p$ be a prime and $q=p^e$, where $e$ is a positive integer. Let $\f_{p^e}$ be the finite field with $p^e$ elements. A polynomial $f \in \Bbb F_{p^e}[{\tt x}]$ is called a \textit{permutation polynomial} (PP) of $\Bbb F_{p^e}$ if the associated mapping $x\mapsto f(x)$ from $\f_{p^e}$ to $\f_{p^e}$ is a permutation of $\Bbb F_{p^e}$. 

It is a well known fact that a function $f:\Bbb F_q\to \Bbb F_q$ is bijective if and only if 
\[
\sum_{a\in\Bbb F_q}f(a)^i
\begin{cases}
=0&\text{if}\ 1\le i\le q-2,\cr
\ne 0&\text{if}\ i=q-1.
\end{cases}
\]

Therefore, an explicit evaluation of the sum $\sum_{a\in\Bbb F_q}f_{n}(a)^i$ for any $1\le i\le q-1$ would provide necessary conditions for $f_{n}$ to be a PP of $\f_q$. We compute the sums $\sum_{a\in \f_q}f_{n}(a)$ and $\sum_{a\in \f_q}f_{n}^2(a)$ in this paper. Dickson polynomials have played a pivotal role in the area of permutation polynomials over finite fields. We point out the connection between Fibonacci polynomials and Dickson polynomials of the second kind (DPSK). We also would like to highlight the fact that the first moment and the second moment of DPSK have never appeared in any literature before. Since the permutation behaviour of DPSK is not completely known yet, we believe that the results concerning $\sum_{a\in \f_q}f_{n}(a)$ and $\sum_{a\in \f_q}f_{n}^2(a)$ would be a motivation for further investigation of $\sum_{a\in\Bbb F_q}f_{n}(a)^i$, where $3\le i\le q-1$, and to help forward the area. 

Here is an overview of the paper.

In Subsection 1.1, we present more properties of Fibonacci polynomials that will be used throughout the paper. In Section 2, we explore self-reciprocal polynomials arising from Fibonacci polynomials over $\mathbb{Z}$ and $\mathbb{Z}_p$, where $p=2$ and $p>5$. We also present some partial results when $p=3$ and $p=5$. 

In Section 3, we discuss the connection between Fibonacci polynomials and Dickson polynomials of the second kind (DPSK) which are related to Chebyshev polynomials. In Section 4, we find necessary condtions for Fibonacci polynomials to be permutation polynomials over finite fields. 

\subsection{Some properties of Fibonacci polynomials}

An explicit expression for $f_{n}(x)$ is given by

\begin{equation}\label{EE1}
f_{n+1}(x)=\displaystyle\sum_{j=0}^{\lfloor \frac{n}{2} \rfloor}\,\binom{n-j}{j}\,x^{n-2j},\,\, \textnormal{for} \,\,n\geq 0.
\end{equation}

Therefore 

\begin{equation}\label{EE2}
f_{n}(x)=\displaystyle\sum_{j=0}^{\lfloor \frac{n-1}{2} \rfloor}\,\binom{n-j-1}{j}\,x^{n-2j-1},\,\, \textnormal{for} \,\,n\geq 0.
\end{equation}

Another explicit formula for $f_n(x)$ is given by 

\begin{equation}\label{EE1.3}
f_n(x)=\displaystyle\sum_{k=0}^{n}\,f(n,k)\,x^k,
\end{equation}

where $f(n,k)=\displaystyle\binom{\frac{n+k-1}{2}}{k}$, and $n$ and $k$ have different parity; see \cite[Section 9.4]{Benjamin-Quinn}. The coefficient $f(n,k)$ can also be thought of as the number of ways of writing $n-1$ as an ordered sum involving only 1 and 2, so that 1 is used exactly $k$ times. 

There is yet another explicit formula for $f_n(x)$:
\begin{equation}\label{EN1.4}
f_n(x)=\displaystyle\frac{\alpha^n(x)-\beta^n(x)}{\alpha(x)-\beta(x)},
\end{equation}

where 

$$\alpha(x)=\displaystyle\frac{x+\sqrt{x^2+4}}{2}\,\,\,\textnormal{and}\,\,\,\beta(x)=\displaystyle\frac{x-\sqrt{x^2+4}}{2}.$$

are the solutions of the quadratic equation $u^2-xu-1=0$. 

The generating function of $f_n(x)$ is given by 
\begin{equation}\label{EE3}
\displaystyle\sum_{n=0}^{\infty}\,f_n(x)\,z^n=\displaystyle\frac{z}{1-xz-z^2};
\end{equation}

see \cite[Chapter 37]{Koshy-2001}.

%%%%%%%%%%%%%%%%%%%%%%%%%%%%%%%%%%
%   section 2
%%%%%%%%%%%%%%%%%%%%%%%%%%%%%%%%%%

\section{Fibonacci Self-reciprocal polynomials}

\subsection{Self-reciprocal polynomials over $\mathbb{Z}$}

In this subsection, we completely classify the self-reciprocal polymomials arising from Fibonacci polynomials over $\mathbb{Z}$. 

\begin{thm}\label{T2.1}
If $f_n$ is self-reciprocal, then $n$ is odd. 
\end{thm}

\begin{proof}
Assume that $n$ is even. Then from \eqref{E2} it is clear that there is no constant term in the polyomial. So $f_n$ cannot be self-reciprocal. 
\end{proof}

\begin{rmk}
The above result is true in any characteristic. 
\end{rmk}

Hereafter we always assume that $n$ is even since we consider the explicit expression of $f_{n+1}$ (Eq.~\eqref{EE1}) in the proofs in Section 2. 

\begin{thm}
$f_n$ is self-reciprocal if and only if $n\in \{3, 5\}$. 
\end{thm}

\begin{proof}
The polynomials $f_3=x^2+1$ and $f_5=x^4+3x^2+1$ are clearly self-reciprocal. We show that $f_n$ is not self-reciprocal when $n\neq 3, 5$. Recall that 

\[
\begin{split}
& f_{n+1}(x)=\displaystyle\sum_{j=0}^{\lfloor \frac{n}{2} \rfloor}\,\binom{n-j}{j}\,x^{n-2j},\,\, \textnormal{for} \,\,n\geq 0.
\end{split}
\]

Note that here $n$ is even since we consider $f_{n+1}(x)$. Let $n$ be even and $n\neq 2,4$. 

\[
\begin{split}
& f_{n+1}(x)= x^{n} + \binom{n-1}{1}x^{n-2}+ \binom{n-2}{2}x^{n-4}+....+\binom{\frac{n}{2} +1}{\frac{n}{2}-1}x^{2}+1.
\end{split}
\]

We show that the coefficients of $x^{n-2}$ and $x^2$ are not equal. Assume that they are the same, i.e.
$$\binom{n-1}{1}=\binom{\frac{n}{2}+1}{\frac{n}{2}-1}.$$
A straightforward computation yields that $n^2 - 6n+8=0$, which implies $n=2\,\textnormal{or}\,4$. This contradicts the fact that $n\neq 2,4$.

\end{proof}
                                                                                                                                                                                                                                                                                                                                                                                                                                                                                                                                                                                                                                                                                                                                                                                                                                                                                                                                                                                                                                                                       
\subsection{In even characteristic}

In this subsection, we consider the even characteristic case. 

\begin{thm}
$f_n$ is self-reciprocal if and only if $n\in \{3, 5\}$. 
\end{thm}

\begin{proof}
$f_3=x^2+1$ and $f_5=x^4+x^2+1$ are clearly self-reciprocal. Now assume that $n\neq 3, 5$. We show that $f_n$ is not self-reciprocal when $n\neq 3, 5$. Recall that 

\[
\begin{split}
& f_{n+1}(x)=\displaystyle\sum_{j=0}^{\lfloor \frac{n}{2} \rfloor}\,\binom{n-j}{j}\,x^{n-2j},\,\, \textnormal{for} \,\,n\geq 0.
\end{split}
\]

Note that 

\begin{center}
$f_{n+1}$ is self-reciprocal if and only if $\binom{n-j}{j}=\binom{\frac{n}{2}+j}{\frac{n}{2}-j}$, for all $0\leq j\leq \frac{n}{2}$. 
\end{center}

We divide the proof into three cases:
\begin{itemize}
\item [(i)] $n\equiv 0\,\textnormal{or}\,6\pmod{8}$. 
\item [(ii)] $n\equiv 2\pmod{8}$. 
\item [(iii)] $n\equiv 4\pmod{8}$. 
\end{itemize}

\noindent \textbf{Case 1.} $n\equiv 0\,\textnormal{or}\,6\pmod{8}$. 

We show that $\binom{\frac{n}{2}+1}{\frac{n}{2}-1}\equiv 0\pmod{2}$, but $\binom{n-1}{1}\not\equiv 0\pmod{2}$. 

$\binom{n-1}{1}=n-1\equiv 5\,\textnormal{or}\,7\pmod{8}$, which implies $\binom{n-1}{1}\not\equiv 0\pmod{2}$. 

Consider $\binom{\frac{n}{2}+1}{\frac{n}{2}-1}=\frac{n(n+2)}{8}$. Since $n\equiv 0\,\textnormal{or}\,6\pmod{8}$, we have $n=8\ell_1$ and $n=8\ell_2$ for some integers $\ell_1>0$ and $\ell_2>0$. Then we have $\binom{\frac{n}{2}+1}{\frac{n}{2}-1}=\frac{n(n+2)}{8}=\ell_1(n+2)\,\,\textnormal{or}\,\,n\ell_2$. In either case, we have  $\binom{\frac{n}{2}+1}{\frac{n}{2}-1}\equiv 0\pmod{2}$ since $n$ is even. 

\noindent \textbf{Case 2.} $n\equiv 2\pmod{8}$ 

We show that the coefficient of $x^6$, $\binom{\frac{n}{2}+3}{\frac{n}{2}-3}$, is congruent to zero modulo $2$, but the coefficient of $x^{n-6}$, $\binom{n-3}{3}$, is not congruent to zero modulo $2$.

\noindent \textbf{Case 2.1} First consider $\binom{n-3}{3}$. We have 
$$\binom{n-3}{3}=\frac{(n-3)(n-4)(n-5)}{6}.$$

Since $n\equiv 2\pmod{8}$, we have $n=2+8\ell_1$, for some integer $\ell_1$, which implies $n-4=2(4\ell_1-1).$ Since $\textnormal{gcd}(3,8)=1$, we have

$$\binom{n-3}{3}=\frac{(n-3)(4\ell_1-1)(n-5)}{3}\equiv 4\ell_1-1\pmod{8},$$

which implies $\binom{n-3}{3}\equiv 1\pmod{2}.$

\noindent \textbf{Case 2.2} Next we consider $\binom{\frac{n}{2}+3}{\frac{n}{2}-3}$. Note that since $n\equiv 2 \pmod{8}$, $\frac{n}{2}\equiv 1\,\textnormal{or}\,5 \pmod{8}$.

\noindent \textbf{Subcase 2.2a.} $\frac{n}{2} \equiv 1 \pmod{8}$. We have

\begin{equation}\label{Dec233}
\binom{\frac{n}{2}+3}{\frac{n}{2}-3} = \frac{(\frac{n}{2}+3)(\frac{n}{2}+2)(\frac{n}{2}+1)(\frac{n}{2})(\frac{n}{2}-1)(\frac{n}{2}-2)}{6 \cdot 5 \cdot 4 \cdot 3 \cdot 2 \cdot 1}.
\end{equation}

Since $\frac{n}{2} \equiv 1 \pmod{8}$, $(\frac{n}{2}-1)=8\ell_1$ and $(\frac{n}{2}+3)=4\ell_2$, for some integers $\ell_1, \ell_2$. From \eqref{Dec233} we have 

$$\binom{\frac{n}{2}+3}{\frac{n}{2}-3} \equiv 0\pmod{2}.$$

\noindent \textbf{Subcase 2.2b.} $\frac{n}{2} \equiv 5 \pmod{8}$. We have

\begin{equation}\label{Dec234}
\binom{\frac{n}{2}+3}{\frac{n}{2}-3} = \frac{(\frac{n}{2}+3)(\frac{n}{2}+2)(\frac{n}{2}+1)(\frac{n}{2})(\frac{n}{2}-1)(\frac{n}{2}-2)}{6 \cdot 5 \cdot 4 \cdot 3 \cdot 2 \cdot 1}.
\end{equation}

Since $\frac{n}{2} \equiv 5 \pmod{8}$, $(\frac{n}{2}-1)=4\ell_1$ and $(\frac{n}{2}+3)=8\ell_2$, for some integers $\ell_1, \ell_2$. From \eqref{Dec234} we have 

$$\binom{\frac{n}{2}+3}{\frac{n}{2}-3} \equiv 0\pmod{2}.$$

\noindent \textbf{Case 3.} $n\equiv 4\pmod{8}$. Note that since $n\equiv 4\pmod{8}$, we have $n\equiv 12\pmod{16}$ or $n\equiv 4\pmod{16}$. 

\noindent \textbf{Subcase 3.1.} $n\equiv 12\pmod{16}$. We show that the coefficient of $x^4$, $\binom{\frac{n}{2}+2}{\frac{n}{2}-2}$, is congruent to zero modulo 2, but the coefficient of $x^{n-4}$, $\binom{n-2}{2}$, is not.

We have 

\[
\begin{split}
\binom{\frac{n}{2}+2}{\frac{n}{2}-2}&=\frac{(\frac{n}{2}+2)(\frac{n}{2}+1)(\frac{n}{2})(\frac{n}{2}-1)}{4\cdot 3\cdot 2\cdot 1}.
\end{split} 
\]

Since $n\equiv 12\pmod{16}$, we have $\frac{n}{2}\equiv 6\pmod{8},$ which implies $(\frac{n}{2}+2)=8\ell_1$ and $\frac{n}{2}=2\ell_2$ for some inetegrs $\ell_1, \ell_2$. Hence $\binom{\frac{n}{2}+2}{\frac{n}{2}-2}\equiv 0\pmod{2}$. 

Consider $\binom{n-2}{2}$. We have 

\[
\begin{split}
\binom{n-2}{2}&=\frac{(n-2)(n-3)}{2}.
\end{split} 
\]

Since $n\equiv 12\pmod{16}$, both $\frac{(n-2)}{2}$ and $(n-3)$ are odd, which implies $\binom{n-2}{2}\equiv 1\pmod{2}$. 

\noindent \textbf{Subcase 3.2.} $n\equiv 4\pmod{16}$. We show that the coefficient of $x^{12}$, $\binom{\frac{n}{2}+6}{\frac{n}{2}-6}$, is congruent to zero modulo 2, but the coefficient of $x^{n-12}$, $\binom{n-6}{6}$, is not.

We have 

\[
\begin{split}
\binom{\frac{n}{2}+6}{\frac{n}{2}-6}&=\frac{\displaystyle\prod_{i=-5}^{6}\,\Big(\frac{n}{2}+i\Big)}{12!}\cr
&=\frac{\displaystyle\prod_{i=-5}^{6}\,\Big(\frac{n}{2}+i\Big)}{2^{10}\cdot 3^5\cdot 5^2\cdot 7\cdot 11}.
\end{split} 
\]

Since $n\equiv 4\pmod{16}$, we have $\frac{n}{2}\equiv 2\pmod{8},$ which implies $(\frac{n}{2}+6)=8\ell_1$, $(\frac{n}{2}+4)=2\ell_2$, $(\frac{n}{2}+2)=4\ell_3$, $\frac{n}{2}=2\ell_4$, $(\frac{n}{2}-2)=8\ell_5$, and $(\frac{n}{2}-4)=2\ell_6$ for some inetegrs $\ell_1, \ell_2, \ell_3, \ell_4, \ell_5, \ell_6$. Hence $\binom{\frac{n}{2}+6}{\frac{n}{2}-6}\equiv 0\pmod{2}$. 

Consider $\binom{n-6}{6}$. We have 

\[
\begin{split}
\binom{n-6}{6}&=\frac{(n-6)(n-7)(n-8)(n-9)(n-10)(n-11)}{6\cdot 5\cdot 4\cdot 3\cdot 2\cdot 1}\cr
&=\frac{(n-6)(n-7)(n-8)(n-9)(n-10)(n-11)}{2^4\cdot 3^2\cdot 5\cdot 1}.
\end{split} 
\]

Since $n\equiv 4\pmod{16}$, $n-6=2\ell_1, n-8=4\ell_2, n-10=2\ell_3$, where $\ell_1, \ell_2, \ell_3$ are odd integers. Also, $n-7, n-9$ and $n-11$ are odd. Hence $\binom{n-6}{6}\equiv 1\pmod{2}$. 

This completes the proof. 

\end{proof}

\subsection{In odd characteristic}

In this subsection, we consider the odd characteristic case. 

\begin{thm}
Let $p$ be a prime and $p>5$. Then $f_n$ is self-reciprocal if and only if $n\in \{3, 5\}$. 
\end{thm}

\begin{proof}
Assume that $p>5$. $f_3=x^2+1$ and $f_5=x^4+3x^2+1$ are clearly self-reciprocal. Now assume that $n\neq 3, 5$. We show that $f_n$ is not self-reciprocal. 
Let $n$ be even and recall that 
\[
\begin{split}
& f_{n+1}(x)=\displaystyle\sum_{j=0}^{\frac{n}{2}}\,\binom{n-j}{j}\,x^{n-2j},\,\, \textnormal{for} \,\,n\geq 0.
\end{split}
\]

Then we have 
\[
\begin{split}
& f_{n+1}(x)= x^{n} + \binom{n-1}{1}x^{n-2}+ \binom{n-2}{2}x^{n-4}+....+\binom{\frac{n}{2} +1}{\frac{n}{2}-1}x^{2}+1.
\end{split}
\]

Recall that 

\begin{center}
$f_{n+1}$ is self-reciprocal if and only if $\binom{n-j}{j}=\binom{\frac{n}{2}+j}{\frac{n}{2}-j}$, for all $0\leq j\leq \frac{n}{2}$. 
\end{center}

We claim that 
$$\binom{n-1}{1}\not \equiv \binom{\frac{n}{2}+1}{\frac{n}{2}-1}\pmod{p}.$$

We first show that if $\binom{n-1}{1}\equiv 0\pmod{p}$, then $\binom{\frac{n}{2}+1}{\frac{n}{2}-1}\not\equiv 0\pmod{p}$.

Assume that $\binom{n-1}{1}\equiv 0\pmod{p}$, which implies $n\equiv 1\pmod{p}$. Since $p>5$,
$$\binom{\frac{n}{2}+1}{\frac{n}{2}-1}=\frac{n(n+2)}{8}\equiv \frac{3}{8}\not\equiv 0\pmod{p}.$$ 

Now we claim the following. 

If $\binom{n-1}{1}\not \equiv 0\pmod{p}$ and $\binom{\frac{n}{2}+1}{\frac{n}{2}-1}\not\equiv 0\pmod{p}$, then 
$$\binom{n-1}{1}\not \equiv \binom{\frac{n}{2}+1}{\frac{n}{2}-1}\pmod{p}.$$

Let $\binom{n-1}{1}\not \equiv 0\pmod{p}$, $\binom{\frac{n}{2}+1}{\frac{n}{2}-1}\not\equiv 0\pmod{p}$, and assume to the contrary 
\begin{equation}\label{Dec1}
\binom{n-1}{1}\equiv \binom{\frac{n}{2}+1}{\frac{n}{2}-1}\pmod{p},
\end{equation}
which implies $n\equiv 2\pmod{p}$ or $n\equiv 4\pmod{p}$. 

We show that when $n\equiv 2\pmod{p}$ or $n\equiv 4\pmod{p}$, $f_n$ is not self-reciprocal. 

Let $n\equiv 2\pmod{p}$. We consider the coefficient of $x^{n-6}$, $\binom{n-3}{3}$, and the coefficient of $x^6$, $\binom{\frac{n}{2}+3}{\frac{n}{2}-3}$. Since $n\equiv 2 \pmod{p}$, $n=2+p\ell,$ where $\ell$ is an even integer. We have 
$$\binom{n-3}{3}=\frac{(n-3)(n-4)(n-5)}{3!}.$$

Note that $\textnormal{gcd}(3!,p)=1$ since $p>5$. We have 
$$(n-3)(n-4)(n-5)=(p\ell-1)(p\ell-2)(p\ell-3)\equiv -6\not\equiv 0\pmod{p}.$$

Hence $\binom{n-3}{3}\not\equiv 0\pmod{p}$. 

Now we show that $\binom{\frac{n}{2}+3}{\frac{n}{2}-3}$ is congruent to zero modulo $p$. We have 
\begin{equation}\label{Dec231}
\binom{\frac{n}{2}+3}{\frac{n}{2}-3}=\frac{(\frac{n}{2}+3)!}{(\frac{n}{2}-3)!\,\,6!}.
\end{equation}

Since $n\equiv 2\pmod{p}$, we have $\frac{n}{2}\equiv 1\pmod{p}$. Since $\textnormal{gcd}(6!,p)=1$ and the right hand side of \eqref{Dec231} contains the term $(\frac{n}{2}-1)$, we have $\binom{\frac{n}{2}+3}{\frac{n}{2}-3}\equiv 0\pmod{p}$. Hence $f_n$ is not self-reciprocal. 

Let $n\equiv 4\pmod{p}$. We consider the coefficient of $x^{n-10}$, $\binom{n-5}{5}$, and the coefficient of $x^{10}$, $\binom{\frac{n}{2}+5}{\frac{n}{2}-5}$. Since $n\equiv 4 \pmod{p}$, $n=4+p\ell,$ where $\ell$ is an even integer. We have 
$$\binom{n-5}{5}=\frac{(n-5)(n-6)(n-7)(n-8)(n-9)}{5!}.$$

Note that $\textnormal{gcd}(5!,p)=1$ since $p>5$. We have 
\[
\begin{split}
(n-5)(n-6)(n-7)(n-8)(n-9)&=(p\ell-1)(p\ell-2)(p\ell-3)(p\ell-4)(p\ell-5)\cr
&\not\equiv 0\pmod{p}.
\end{split}
\]

Hence $\binom{n-5}{5}\not\equiv 0\pmod{p}$. 

Now we show that $\binom{\frac{n}{2}+5}{\frac{n}{2}-5}$ is congruent to zero modulo $p$. We have 
\begin{equation}\label{Dec232}
\binom{\frac{n}{2}+5}{\frac{n}{2}-5}=\frac{(\frac{n}{2}+5)!}{(\frac{n}{2}-5)!\,\,10!}.
\end{equation}

Since $n\equiv 4\pmod{p}$, we have $\frac{n}{2}\equiv 2\pmod{p}$. 

\noindent \textbf{Case 1.} $p>7$. Since $\textnormal{gcd}(10!,p)=1$ and the right hand side of \eqref{Dec232} contains the term $(\frac{n}{2}-2)$, we have $\binom{\frac{n}{2}+5}{\frac{n}{2}-5}\equiv 0\pmod{p}$. 

\noindent \textbf{Case 2.} $p=7$. Since the term $(\frac{n}{2}+5)$ in the numerator on the right hand side of \eqref{Dec232} is a multiple of $7$ and it contains the term $(\frac{n}{2}-2)$, we have $\binom{\frac{n}{2}+5}{\frac{n}{2}-5}\equiv 0\pmod{p}$. 

Hence $f_n$ is not self-reciprocal. 

\end{proof}

\begin{rmk}
Let $p=3$ and $\l\geq 0$ be an integer. Then $f_n$ is self-reciprocal if $n$ satisfies one of the follwing:
\begin{enumerate}
\item [(i)] $n=p^{\l}$,  
\item [(ii)] $n=5\cdot p^{\l}$, 
\item [(iii)] $n=41 \cdot p^{\l}$,
\item [(iv)] $n=5\cdot73\cdot p^{\l}$,
\item [(v)] $n=5^{2}\cdot1181 \cdot p^{\l}$,
\end{enumerate}

\vskip 0.1in

Let $p=5$ and $\l\geq 0$ be an integer. Then $f_n$ is self-reciprocal if $n$ satisfies one of the follwing:
\begin{enumerate}
\item [(i)] $n=p^{\l}$, 
\item [(ii)] $n=3\cdot p^{\l}$, 
\item [(iii)] $n=13 \cdot p^{\l}$,
\item [(iv)] $n=3\cdot 29 \cdot p^{\l}$,
\item [(v)] $n=3^{2}\cdot7 \cdot p^{\l}$,
\end{enumerate}
\end{rmk}

\begin{rmk}
\textnormal{Our numerical results obtained from the computer indicate that the number of patterns of $n$ increases as $n$ increases. So it would be an arduous task to find necessary and sufficient conditions on $n$ for $f_n$ to be a self-reciprocal when $p=3$ and $p=5$. }
\end{rmk}
                                                                                                                                                                                                                                                                                                                                                                                                                                                                                                                                                                                                                                                                                                                                                                                                                                                                                                                                                                                                                                                                       
%%%%%%%%%%%%%%%%%%%%%%%%%%%%%%%%%%%%
%   section 5
%%%%%%%%%%%%%%%%%%%%%%%%%%%%%%%%%%%%

\section{Fibonacci Permutation Polynomials}

\subsection{Fibonacci polynomials and Dickson polynomials}

Dickson polynomials have played a pivotal role in the study of permutation polynomials over finite fields. 

The $n$-th Dickson polynomial of the second kind $E_n(x,a)$ is defined by
\[
E_{n}(x,a) = \sum_{i=0}^{\lfloor\frac n2\rfloor}\dbinom{n-i}{i}(-a)^{i}x^{n-2i},
\]
where $a\in \f_q$ is a parameter; see \cite[Chapter 2]{Lidl-Mullen-Turnwald-1993}. 

Dickson polynomials of the second kind (DPSK) are closely related to the well-known Chebyshev polynomials over the complex numbers by
$$E_n(2x,1)=U_n(x),$$

where $U_n(x)$ is the Chebyshev polynonmial of degree $n$ of the second kind. 

Let $x=u+\frac{a}{u}$ and $u-\frac{a}{u}\neq 0$, i.e. $u\neq \pm \sqrt{a}$. Note that $x=u+\frac{a}{u}$ implies $u^2-xu-1=0$, where $u$ is in the extension field $\f_{q^2}$. Then the functional expression of DPSK is given by 
\begin{equation}\label{E2.1}
E_n(u+\frac{a}{u}, a)=\displaystyle\frac{u^{n+1}-(\frac{a}{u})^{n+1}}{u-\frac{a}{u}}.
\end{equation}

For $u=\sqrt{a}$ and $u=-\sqrt{a}$ we have 
$$E_n(2\sqrt{a}, a)=(n+1)(\sqrt{a})^n$$
and
$$E_n(-2\sqrt{a}, a)=(n+1)(-\sqrt{a})^n;$$
see \cite[Chapter 2]{Lidl-Mullen-Turnwald-1993}. 

We note to the reader that when $a=-1$, $n$-th Dickson polynomial of the second kind is the $(n+1)$-st Fibonacci polynomial. Therefore the functional expression of Fibonacci polynomials is given by 
\begin{equation}\label{EEE2.2}
f_{n+1}(u-\frac{1}{u})=\displaystyle\frac{u^{n+1}-(\frac{-1}{u})^{n+1}}{u+\frac{1}{u}}
\end{equation}

with the condition that $u\neq \pm b$ if $b^2=-1$ for some $b\in \f_q$. When $u=\pm b$ with $b^2=-1$ for some $b\in \f_q$, i.e. $x^2+4=0$, we have 
\begin{equation}\label{EEE2.3}
f_{n+1}(\pm 2b)=(n+1)(\pm b)^n.
\end{equation}

The permutation behaviour of DPSK is not completely known yet, but it has been studied by many authors to a large degree: Stephen D. Cohen, Rex Matthews, Mihai Cipu, Marie Henderson and Robert  Coulter, to name a few. We refer the reader to \cite{Cipu-2004, Cipu-Cohen-2008, Cohen-1994, Cohen-1993, Coulter-Matthews-2002, Henderson-1997, Henderson-Matthews-1998, Henderson-Matthews-1995, Lidl-Mullen-Turnwald-1993, Matthews-Thesis-1982} for more details about DPSK and their permutation behaviour over finite fields. 

%%%%%%%%%%%%%%%%%%%%%%%%%%%%%%%%%%
%   section 6
%%%%%%%%%%%%%%%%%%%%%%%%%%%%%%%%%%

\section{When is $f_{n}=f_{m}$?}

In \cite{Wang-Yucas-FFA-2012}, Wang and Yucas introduced the $n$-th Dickson polynomial of the $(k+1)$-th kind and the $n$-th reversed Dickson polynomial of the $(k+1)$-th kind.

For $a\in \f_q$, the $n$-th Dickson polynomial of the $(k+1)$-th kind $D_{n,k}(x,a)$ is defined by 
\[
D_{n,k}(x,a) = \sum_{i=0}^{\lfloor\frac n2\rfloor}\frac{n-ki}{n-i}\dbinom{n-i}{i}(-a)^{i}x^{n-2i},
\]

and $D_{0,k}(x,a)=2-k$. 

When $a=1$, Wang and Yucas showed that the sequence of Dickson polynomials of the $(k+1)$-th kind in terms of degrees modulo $x^q-x$ is a periodic function with period $2c$, where $c=\frac{p(q^2-1)}{4}$.

\subsection{$p\equiv 1\pmod{4}$ with any $e$ or $p\equiv 3\pmod{4}$ with even $e$.} 

When $p\equiv 1\pmod{4}$ with any $e$ or $p\equiv 3\pmod{4}$ with even $e$, since $-1$ is a square in $\f_{p^e}$, then the following theorem follows from \cite[Theorem 2.12]{Wang-Yucas-FFA-2012}. This result also appeared in \cite{Henderson-Matthews-1995} .

\begin{thm}\label{TD261}
Let $e$ be even or $p\equiv 1\pmod{4}$ with odd $e$. If $n_1\equiv n_2\pmod{\frac{p(p^{2e}-1)}{2}}$, then $f_{n_1}=f_{n_2}$ for all $x\in \f_{p^e}$. 
\end{thm}

We would like to point out to the reader that the results in Theorems \ref{TD261}, \ref{TD262}, and \ref{TD2} also follow from the sign class argument that apeared in \cite{Matthews-Thesis-1982}. For the completeness, we give proofs for the following two theorems. 

\subsection{$p\equiv 3\pmod{4}$ and $e$ is odd.} From \eqref{EN1.4} we have 
$$f_n(u-\frac{1}{u})=\displaystyle\frac{u_1^n-u_2^n}{u_1-u_2},$$ where $u_1=\displaystyle\frac{x+\sqrt{x^2+4}}{2}$ and  $u_2=\displaystyle\frac{x-\sqrt{x^2+4}}{2}$ are the solutions of the quadratic equation $u^2-xu-1=0$. Here $u$ is in the extension field $\f_{p^{2e}}$. 

We note to the reader that $u_1 \neq u_2$ when $p\equiv 3\pmod{4}$ and $e$ is odd since $u_1=u_2$ implies $x^2=-4$, which is not true when $p\equiv 3\pmod{4}$ and $e$ is odd. Thus we have the following result. 

\begin{thm}\label{TD262}
Let $p\equiv 3\pmod{4}$ and $e$ be odd. If $n_1\equiv n_2\pmod{p^{2e}-1}$, then $f_{n_1}=f_{n_2}$ for all $x\in \f_{p^e}$. 
\end{thm}

\begin{proof}
For $x\in \f_{p^e}$, there exists $u\in \f_{p^{2e}}$ such that $x=u-\frac{1}{u}$.
Then we have
\[
\begin{split}
f_{n_1}(x)&=\displaystyle\frac{u_1^{n_1}-u_2^{n_1}}{u_1-u_2}\cr
&=\displaystyle\frac{u_1^{n_2}-u_2^{n_2}}{u_1-u_2}\cr
&=f_{n_2}(x).
\end{split}
\]
\end{proof}

\subsection{$p=2$.} 

\begin{thm}\label{TD2}
If $n_{1}\equiv n_{2}\,\, \pmod{2^{{2e+1}}-2}$, then $f_{n_1}=f_{n_2}$.
\end{thm}

\begin{proof}
Let $n_{1}\equiv n_{2}\,\, \pmod{2^{{2e+1}}-2}$. Note that, $2^{{2e+1}}-2=(2^{2e}-1)+(2^{2e}-1)$. When $u\neq1$, i.e. $x\neq0$, we have 

\[
\begin{split}
f_{n_1}(x)&=\displaystyle\frac{u^{n_1}+\Big(\frac{1}{u}\Big)^{n_1}}{u+\frac{1}{u}} \cr
&= \displaystyle\frac{u^{n_2}+\Big(\frac{1}{u}\Big)^{n_2}}{u+\frac{1}{u}} \cr
&= f_{n_2}(x)
\end{split}
\]

When $u=1$, i.e. $x=0$, we have in characteristic 2, 

\[
\begin{split}
f_{n_1}(x)=n_{1} =n_{2}=f_{n_2}(x)
\end{split}
\]

Thus for all ${x\in \f_{2^e}}$, $f_{n_1}(x)=f_{n_2}(x)$.

\end{proof}

%%%%%%%%%%%%%%%%%%%%%%%%%%%%%%%%%%%%
%   section 7
%%%%%%%%%%%%%%%%%%%%%%%%%%%%%%%%%%%%

\section{Computation of $\sum_{a\in \f_q}f_{n}(a)$. }

In this Section, we compute the sums $\sum_{a\in \f_q}f_{n}(a)$ and $\sum_{a\in \f_q}f_{n}^2(a)$. We would also like to emphasize again that the sums $\sum_{a\in \f_q}f_{n}(a)$ and $\sum_{a\in \f_q}f_{n}^2(a)$, which provide necessary conditions for $f_{n}$ to be a PP of $\f_q$, have never appeared in any literature about DPSK before. We believe that the results of this section would help forward the area. 

\subsection{$q$ odd} In this subsection, we compute $\sum_{a\in \f_q}f_{n}(a)$ when $q$ is odd. 

By \eqref{EE3}, we have
\begin{equation}\label{E4.1}
\begin{split}
\displaystyle\sum_{n=0}^{\infty}\,f_{n}(x)\,z^n&=\displaystyle\frac{z}{1-xz-z^2}\cr
&=\displaystyle\frac{z}{1-z^2}\,\,\displaystyle\frac{1}{1+(\frac{z}{z^2-1})\,x}\cr
&=\displaystyle\frac{z}{1-z^2}\,\,\displaystyle\sum_{k\geq 0} \Big(\frac{z}{z^2-1}\Big)^k\,\,(-1)^k\,\,x^k\cr
&=\displaystyle\frac{z}{1-z^2}\,\,\Big[1+\displaystyle\sum_{k=1}^{q-1}\displaystyle\sum_{l\geq 0} \Big(\frac{z}{z^2-1}\Big)^{k+l(q-1)}\,\,(-1)^{k+l(q-1)}\,\,x^{k+l(q-1)} \Big]\cr
&\equiv \displaystyle\frac{z}{1-z^2}\,\,\Big[1+\displaystyle\sum_{k=1}^{q-1}\displaystyle\sum_{l\geq 0} \Big(\frac{z}{z^2-1}\Big)^{k+l(q-1)}\,\,(-1)^{k}\,\,x^k \Big]\,\,\,\pmod{x^q-x}\cr
&=\displaystyle\frac{z}{1-z^2}\,\,\Big[1+\displaystyle\sum_{k=1}^{q-1}\,\,(-1)^k\,\,\displaystyle\frac{(\frac{z}{z^2-1})^{k}}{1-(\frac{z}{z^2-1})^{q-1}}\,\,x^k \Big]\cr
&=\displaystyle\frac{z}{1-z^2}\,\,\Big[1+\displaystyle\sum_{k=1}^{q-1}\,\,(-1)^k\,\,\displaystyle\frac{(z^2-1)^{q-1-k}\,\,z^{k}}{(z^2-1)^{q-1} - z^{q-1}}\,\,x^k \Big]\cr
\end{split}
\end{equation}

\subsection{The case $p\equiv 3\pmod{4}$ and $e$ is odd.}  

Since $f_{n_1}\equiv f_{n_2} \pmod{x^q-x}$ when $n_1, n_2 >0$ and $n_1\equiv n_2 \pmod{q^2-1}$, we have the following. 
\begin{equation}\label{E4.2}
\begin{split}
\displaystyle\sum_{n\geq 0} \,f_{n}\,z^n&= \displaystyle\sum_{n\geq 1} \,f_{n}\,z^n\cr
&=\displaystyle\sum_{n=1}^{q^2-1}\,\, \displaystyle\sum_{l\geq 0} \,f_{n+l(q^2-1)}\,z^{n+l(q^2-1)}\cr
&\equiv \displaystyle\sum_{n=1}^{q^2-1}\,\,f_n\,\, \displaystyle\sum_{l\geq 0}\,z^{n+l(q^2-1)}\,\,\pmod{x^q-x}\cr
&=\displaystyle\frac{1}{1-z^{q^2-1}} \displaystyle\sum_{n=1}^{q^2-1}\,f_n\,z^n
\end{split}
\end{equation}

Combining \eqref{E4.1} and \eqref{E4.2} gives

\[
\begin{split}
\displaystyle\frac{1}{1-z^{q^2-1}} \displaystyle\sum_{n=1}^{q^2-1}\,f_n\,z^n \equiv \displaystyle\frac{z}{1-z^2}\,\,\Big[1+\displaystyle\sum_{k=1}^{q-1}\,\,(-1)^k\,\,\displaystyle\frac{(z^2-1)^{q-1-k}\,\,z^{k}}{(z^2-1)^{q-1} - z^{q-1}}\,\,x^k\Big]\,\,\pmod{x^q-x},
\end{split}
\]

i.e.

\begin{equation}\label{E1}
\begin{split}
\displaystyle\sum_{n=1}^{q^2-1}\,f_n\,z^n \equiv \displaystyle\frac{z\,(z^{q^2-1}-1)}{z^2-1} +\,\,h(z)\,\,\displaystyle\sum_{k=1}^{q-1}\,\,(-1)^k\,\,(z^2-1)^{q-1-k}\,\,z^{k}\,\,x^k \,\,\pmod{x^q-x},
\end{split}
\end{equation}

where
$$h(z)=\displaystyle\frac{z\,(z^{q^2-1}-1)}{(z^2-1)\,[(z^2-1)^{q-1}-z^{q-1}]}.$$

Note that 

\[
\begin{split}
h(z) &= \displaystyle\frac{z\,(z^{q^2-1}-1)}{(z^2-1)\,[(z^2-1)^{q-1}-z^{q-1}]}\cr
&= \displaystyle\frac{z\,(z^{q^2-1}-1)}{(z^2-1)^{q}-z^{q-1}(z^2-1)}\cr
&= \displaystyle\frac{z\,(z^{q^2-1}-1)}{(z^{q-1}-1)\,(z^{q+1}+1)}\cr
&=  \displaystyle\frac{- z\,(z^{q^2}-z)}{(z-z^{q})\,(z^{q+1}+1)}\cr
&= \displaystyle\frac{z\,(1+(z-z^q)^{q-1})}{(z^{q+1}+1)}.
\end{split}
\]

Let $\displaystyle\sum_{k=1}^{q^2-q+1}\,b_kz^k = z\,(1+(z-z^q)^{q-1})$. 

Write $k=\alpha + \beta q$ where $0\leq \alpha, \beta \leq q-1$. Then we have the following.

$$
b_k = \left\{
        \begin{array}{ll}
           (-1)^{\beta} \,\binom{q-1}{\beta} &  \textnormal{if}\,\,\alpha +\beta =q,\\[0.3cm]
            1 &  \textnormal{if}\,\,\alpha +\beta =1, \\[0.3cm]
            0 &  \textnormal{otherwise}.
        \end{array}
    \right.
$$

Summing the above equation as $x$ runs over $\f_q$ in \eqref{E1} we have 

\[
\begin{split}
\displaystyle\sum_{n=1}^{q^2-1}\,\Big(\displaystyle\sum_{x\in \f_q}\,f_n(x)\Big)\,z^n \equiv h(z)\,\,\displaystyle\sum_{k=1}^{q-1}\,\,(-1)^k\,\,(z^2-1)^{q-1-k}\,\,z^{k}\,\,\Big(\displaystyle\sum_{x\in \f_q}\,x^k\Big)\,\,\pmod{x^q-x},
\end{split}
\]

which implies

\begin{equation}\label{E2}
\begin{split}
\displaystyle\sum_{n=1}^{q^2-1}\, \Big(\displaystyle\sum_{x\in \f_q}\,f_n(x)\Big) z^n &=\,\,- h(z)\,\,z^{q-1}.
\end{split}
\end{equation}

From \eqref{E2} we have

\[
\begin{split}
\displaystyle\sum_{n=1}^{q^2-1}\, \Big(\displaystyle\sum_{x\in \f_q}\,f_n(x)\Big) z^n &=\,\,-\,\,z^{q-1}\,\displaystyle\sum_{k=1}^{q^2-q+1}\,b_kz^k\, \displaystyle\frac{1}{(z^{q+1}+1)}.
\end{split}
\]

Now let $d_n=\displaystyle\sum_{x\in \f_q}\,f_n(x)$. Then the above equation can be written as follows. 

\begin{equation}\label{E3}
\begin{split}
(z^{q+1}+1)\,\,\displaystyle\sum_{n=1}^{q^2-1}\,d_n z^n &=\,\,-\,\,z^{q-1}\,\displaystyle\sum_{k=1}^{q^2-q+1}\,b_kz^k.
\end{split}
\end{equation}

\begin{prop} By comparing the coefficient of $z^i$ on both sides of \eqref{E3}, we have the following.

Case 1. When $q=3$. 
\[
\begin{split}
& d_j=0 \,  \  \textnormal{if}  \ 1\leq j \leq q-1; \\
& d_j=-b_{j-(q-1)} \  \textnormal{if}  \ q\leq j \leq q+1; \\
& d_j=-b_{j+2}  \ \textnormal{if}  \ j=q^{2}-(q+1); \\
& d_j=0 \ \textnormal{if}  \ j \geq q^{2}-q.
\end{split}
\]

Case 2. When $q>3$. 
\[
\begin{split}
& d_j=0 \,  \  \textnormal{if}  \ 1\leq j \leq q-1; \\
& d_j=-b_{j-(q-1)} \  \textnormal{if}  \ q\leq j \leq q+1; \\
& d_j=-b_{j-(q-1)} - d_{j-(q+1)} \ \textnormal{if}  \ q+2\leq j \leq {q^2}-(q+2); \\
& d_j=-b_{j+2}  \ \textnormal{if}  \ j=q^{2}-(q+1); \\
& d_j=0 \ \textnormal{if}  \ j \geq q^{2}-q.
\end{split}
\]
\end{prop}

The following theorem is an immediate consequence of the above Proposition and the fact that $d_n=\displaystyle\sum_{x\in \f_q}\,f_n(x)$.

\begin{thm}

Let $b_k$ be defined as in \eqref{E3} for $1\leq k\leq q^2+q-1$. Then we have the following. 

Case 1. When $q=3$. 
\[
\begin{split}
& \displaystyle\sum_{x\in \f_q}\,f_j(x)= 0 \,\,  \  \textnormal{if}  \ 1\leq j \leq q-1; \\
& \displaystyle\sum_{x\in \f_q}\,f_j(x)=-b_{j-(q-1)} \,\,\  \textnormal{if}  \ q\leq j \leq q+1;\\
& \displaystyle\sum_{x\in \f_q}\,f_j(x)=-b_{j+2}  \,\, \ \textnormal{if}  \ j=q^{2}-(q+1);\\
& \displaystyle\sum_{x\in \f_q}\,f_j(x)=0 \,\, \ \textnormal{if}  \ j \geq q^{2}-q.
\end{split}
\]

Case 2. When $q>3$. 
\[
\begin{split}
& \displaystyle\sum_{x\in \f_q}\,f_j(x)= 0 \,\,  \  \textnormal{if}  \ 1\leq j \leq q-1; \\
& \displaystyle\sum_{x\in \f_q}\,f_j(x)=-b_{j-(q-1)} \,\,\  \textnormal{if}  \ q\leq j \leq q+1;\\
& \displaystyle\sum_{x\in \f_q}\,f_j(x)=-b_{j-(q-1)} - \displaystyle\sum_{x\in \f_q}\,f_{j-(q+1)}(x) \,\,\ \textnormal{if}  \ q+2\leq j \leq {q^2}-(q+2);\\
& \displaystyle\sum_{x\in \f_q}\,f_j(x)=-b_{j+2}  \,\, \ \textnormal{if}  \ j=q^{2}-(q+1);\\
& \displaystyle\sum_{x\in \f_q}\,f_j(x)=0 \,\, \ \textnormal{if}  \ j \geq q^{2}-q.
\end{split}
\]

\end{thm}

\subsection{The case $p\equiv 1\pmod{4}$ with any $e$ or $p\equiv 3\pmod{4}$ with even $e$.} 

Let $s=\frac{p(q^2-1)}{2}$.

Since $f_{n_1}\equiv f_{n_2} \pmod{x^q-x}$ when $n_1, n_2 >0$ and $n_1\equiv n_2 \pmod{s}$, we have the following from a computation similar to that of  \eqref{E4.2}.
\begin{equation}\label{E4.2N}
\begin{split}
\displaystyle\sum_{n\geq 0} \,f_{n}\,z^n&=\displaystyle\frac{1}{1-z^{s}} \displaystyle\sum_{n=1}^{s}\,f_n\,z^n
\end{split}
\end{equation}

Combining \eqref{E4.1} and \eqref{E4.2N} gives

\[
\begin{split}
\displaystyle\frac{1}{1-z^{s}} \displaystyle\sum_{n=1}^{s}\,f_n\,z^n \equiv \displaystyle\frac{z}{1-z^2}\,\,\Big[1+\displaystyle\sum_{k=1}^{q-1}\,\,(-1)^k\,\,\displaystyle\frac{(z^2-1)^{q-1-k}\,\,z^{k}}{(z^2-1)^{q-1} - z^{q-1}}\,\,x^k\Big]\,\,\pmod{x^q-x},
\end{split}
\]

i.e.

\begin{equation}\label{E4}
\begin{split}
\displaystyle\sum_{n=1}^{s}\,f_n\,z^n \equiv \displaystyle\frac{z\,(z^{s}-1)}{z^2-1} +\,\,h(z)\,\,\displaystyle\sum_{k=1}^{q-1}\,\,(-1)^k\,\,(z^2-1)^{q-1-k}\,\,z^{k}\,\,x^k \,\,\pmod{x^q-x},
\end{split}
\end{equation}

where
$$h(z)=\displaystyle\frac{z\,(z^{s}-1)}{(z^2-1)\,[(z^2-1)^{q-1}-z^{q-1}]}.$$

Note that 

\[
\begin{split}
h(z) &= \displaystyle\frac{z\,(z^{s}-1)}{(z^2-1)\,[(z^2-1)^{q-1}-z^{q-1}]}\cr
&= \displaystyle\frac{z\,(z^{s}-1)}{(z^2-1)^{q}-z^{q-1}(z^2-1)}\cr
&= \displaystyle\frac{z\,(z^{s}-1)}{(z^{q-1}-1)\,(z^{q+1}+1)}\cr
&=  \displaystyle\frac{- z\,(z^{s+1}-z)}{(z-z^{q})\,(z^{q+1}+1)}
\end{split}
\]

Summing the above equation as $x$ runs over $\f_q$ in \eqref{E4} we have 

\[
\begin{split}
\displaystyle\sum_{n=1}^{s}\,\Big(\displaystyle\sum_{x\in \f_q}\,f_n(x)\Big)\,z^n \equiv h(z)\,\,\displaystyle\sum_{k=1}^{q-1}\,\,(-1)^k\,\,(z^2-1)^{q-1-k}\,\,z^{k}\,\,\Big(\displaystyle\sum_{x\in \f_q}\,x^k\Big)\,\,\pmod{x^q-x},
\end{split}
\]

which implies

\begin{equation}\label{E5}
\begin{split}
\displaystyle\sum_{n=1}^{s}\, \Big(\displaystyle\sum_{x\in \f_q}\,f_n(x)\Big) z^n &=\,\,- h(z)\,\,z^{q-1}.
\end{split}
\end{equation}

From \eqref{E5} we have

\begin{equation}\label{March25}
\begin{split}
\displaystyle\sum_{n=1}^{s}\, \Big(\displaystyle\sum_{x\in \f_q}\,f_n(x)\Big) z^n &=\displaystyle\frac{z^q\,(z^{s+1}-z)}{(z-z^{q})\,(z^{q+1}+1)}.
\end{split}
\end{equation}

Now let $d_n=\displaystyle\sum_{x\in \f_q}\,f_n(x)$ and $\displaystyle\sum_{k=1}^{s-q+2}\,b_kz^k=\displaystyle\frac{z^{s+1}-z}{z^{q-1}-1}$. 
Then we have the following.

$$
b_k = \left\{
        \begin{array}{ll}
           1 &  \textnormal{if}\,\,k=1+t(q-1),\,\,\,\,\textnormal{where}\,\,\,\,0\leq t\leq \frac{p(q+1)}{2}-1,\\[0.3cm]
            0 &  \textnormal{otherwise}.
        \end{array}
    \right.
$$

Now \eqref{March25} can be written as follows. 

\begin{equation}\label{E6}
\begin{split}
(z^{q+1}+1)\,\,\displaystyle\sum_{n=1}^{s}\,d_n z^n &=-z^{q-1}\,\displaystyle\sum_{k=1}^{s-q+2}\,b_kz^k.
\end{split}
\end{equation}

\begin{prop} By comparing the coefficient of $z^i$ on both sides of \eqref{E6}, we have the following.

\[
\begin{split}
& d_j=0 \,  \  \textnormal{if}  \ 1\leq j \leq q-1; \\
& d_j=-b_{j-(q-1)} \  \textnormal{if}  \ q\leq j \leq q+1; \\
& d_j=-b_{j-(q-1)} - d_{j-(q+1)} \ \textnormal{if}  \ q+2\leq j \leq s-(q+1); \\
& d_j=-b_{j+2}  \ \textnormal{if}  \ j=s-q; \\
& d_j=0 \ \textnormal{if}  \ j \geq s-(q-1).
\end{split}
\]
\end{prop}

The following theorem is an immediate consequence of the above Proposition and the fact that $d_n=\displaystyle\sum_{x\in \f_q}\,f_n(x)$.

\begin{thm} Let $f_n$ be the $n$-th Fibonacci polynomial. Assume that $p\equiv 1\pmod{4}$ with any $e$ or $p\equiv 3\pmod{4}$ with even $e$. Let $q=p^e$. Then we have the following. 

\[
\begin{split}
& \displaystyle\sum_{x\in \f_q}\,f_j(x)= 0 \,  \  \textnormal{if}  \ 1\leq j \leq q-1; \\
& \displaystyle\sum_{x\in \f_q}\,f_j(x)=-b_{j-(q-1)} \  \textnormal{if}  \ q\leq j \leq q+1;\\
& \displaystyle\sum_{x\in \f_q}\,f_j(x)=-b_{j-(q-1)} - \displaystyle\sum_{x\in \f_q}\,f_{j-(q+1)}(x) \ \textnormal{if}  \ q+2\leq j \leq s-(q+1);\\
& \displaystyle\sum_{x\in \f_q}\,f_j(x)=-b_{j+2}  \ \textnormal{if}  \ j=s-q;\\
& \displaystyle\sum_{x\in \f_q}\,f_j(x)=0 \ \textnormal{if}  \ j \geq s-(q-1).
\end{split}
\]

\end{thm}

\subsection{$q$ even} 

In this subsection, we assume that $q$ is even and compute the sum $\sum_{a\in \f_q}f_{n}(a)$. When $q$ is even, a computation similar to \eqref{E4.1} shows that 
\begin{equation}\label{NE4.1}
\begin{split}
\displaystyle\sum_{n=0}^{\infty}\,f_{n}(x)\,z^n&\equiv \displaystyle\frac{z}{1-z^2}\,\,\Big[1+\displaystyle\sum_{k=1}^{q-1}\,\,\displaystyle\frac{(z^2-1)^{q-1-k}\,\,z^{k}}{(z^2-1)^{q-1} - z^{q-1}}\,\,x^{k} \Big]\,\,\,\pmod{x^q-x}
\end{split}
\end{equation}

Since $f_{n_1}\equiv f_{n_2} \pmod{x^q-x}$ when $n_1, n_2 >0$ and $n_1\equiv n_2 \pmod{2q^2-2}$, we have the following. 
\begin{equation}\label{NE4.2}
\begin{split}
\displaystyle\sum_{n\geq 0} \,f_{n}\,z^n&= \displaystyle\sum_{n\geq 1} \,f_{n}\,z^n\cr
&=\displaystyle\sum_{n=1}^{2q^2-2}\,\, \displaystyle\sum_{l\geq 0} \,f_{n+l(2q^2-2)}\,z^{n+l(2q^2-2)}\cr
&\equiv \displaystyle\sum_{n=1}^{2q^2-2}\,\,f_n\,\, \displaystyle\sum_{l\geq 0}\,z^{n+l(2q^2-2)}\,\,\pmod{x^q-x}\cr
&=\displaystyle\frac{1}{1-z^{2q^2-2}} \displaystyle\sum_{n=1}^{2q^2-2}\,f_n\,z^n
\end{split}
\end{equation}

Combining \eqref{NE4.1} and \eqref{NE4.2} gives

\[
\begin{split}
\displaystyle\frac{1}{1-z^{2q^2-2}} \displaystyle\sum_{n=1}^{2q^2-2}\,f_n\,z^n \equiv \displaystyle\frac{z}{1-z^2}\,\,\Big[1+\displaystyle\sum_{k=1}^{q-1}\,\,\displaystyle\frac{(z^2-1)^{q-1-k}\,\,z^{k}}{(z^2-1)^{q-1} - z^{q-1}}\,\,x^{k}\Big]\,\,\pmod{x^q-x},
\end{split}
\]

i.e.

\begin{equation}\label{NE1}
\begin{split}
\displaystyle\sum_{n=1}^{2q^2-2}\,f_n\,z^n \equiv \displaystyle\frac{z\,(z^{2q^2-2}-1)}{z^2-1} +\,\,h(z)\,\,\displaystyle\sum_{k=1}^{q-1}\,\,(z^2-1)^{q-1-k}\,\,z^{k}\,\,x^{k} \,\,\pmod{x^q-x},
\end{split}
\end{equation}

where
$$h(z)=\displaystyle\frac{z\,(z^{2q^2-1}-1)}{(z^2-1)\,[(z^2-1)^{q-1}-z^{q-1}]}.$$

Note that 

\[
\begin{split}
h(z) &= \displaystyle\frac{z\,(z^{2q^2-1}-1)}{(z^2-1)\,[(z^2-1)^{q-1}-z^{q-1}]}\cr
&= \displaystyle\frac{z\,(z^{2q^2-1}-1)}{(z^2-1)^{q}-z^{q-1}(z^2-1)}\cr
&= \displaystyle\frac{z\,(z^{2q^2-1}-1)}{(z^{q-1}-1)\,(z^{q+1}+1)}\cr
&=  \displaystyle\frac{- z\,(z^{2q^2}-z)}{(z-z^{q})\,(z^{q+1}+1)}\cr
&= \displaystyle\frac{(z^{2q^2-2}-z)}{(z^{q-1}-1)\,(z^{q+1}+1)}.
\end{split}
\]

\medskip Summing the above equation as $x$ runs over $\f_q$ in \eqref{NE1} we have 

\[
\begin{split}
\displaystyle\sum_{n=1}^{2q^2-2}\,\Big(\displaystyle\sum_{x\in \f_q}\,f_n(x)\Big)\,z^n \equiv h(z)\,\,\displaystyle\sum_{k=1}^{q-1}\,\,(z^2-1)^{q-1-k}\,\,z^{k}\,\,\Big(\displaystyle\sum_{x\in \f_q}\,x^{k}\Big)\,\,\pmod{x^q-x},
\end{split}
\]

which implies

\begin{equation}\label{NE2}
\begin{split}
\displaystyle\sum_{n=1}^{2q^2-2}\, \Big(\displaystyle\sum_{x\in \f_q}\,f_n(x)\Big) z^n &=\,\,h(z)\,\,z^{q-1}.
\end{split}
\end{equation}

From \eqref{E2} we have

\[
\begin{split}
\displaystyle\sum_{n=1}^{2q^2-2}\, \Big(\displaystyle\sum_{x\in \f_q}\,f_n(x)\Big) z^n &=\,\,\,\,\displaystyle\frac{(z^{2q^2+q-3}-z^q)}{(z^{q-1}-1)\,(z^{q+1}+1)}\,.
\end{split}
\]

Now let $d_n=\displaystyle\sum_{x\in \f_q}\,f_n(x)$. Then the above equation can be written as follows. 

\begin{equation}\label{NE3}
\begin{split}
(z^{q-1}-1)\,(z^{q+1}+1)\,\,\displaystyle\sum_{n=1}^{2q^2-2}\,d_n z^n &=\,\,\,\,{(z^{2q^2+q-3}-z^q)}\,.
\end{split}
\end{equation}

\begin{prop} By comparing the coefficient of $z^i$ on both sides of \eqref{E3}, we have the following.

\[
\begin{split}
& d_j=1+ d_{q}\  \textnormal{if}  \ j=1; \\
& d_j=1+ d_{2q^2-4}-d_{2q^2-q-3}\  \textnormal{if}  \ j=2q^2-2;\\
& d_j=0\  \textnormal{otherwise}.
\end{split}
\]
\end{prop}

The following theorem is an immediate consequence of the above Proposition and the fact that $d_n=\displaystyle\sum_{x\in \f_q}\,f_n(x)$.

\begin{thm}\label{NT2.6}

Let $b_k$ be defined as in \eqref{NE3} for $1\leq k\leq q^2+q-1$. Then we have the following. 

\[
\begin{split}
& \displaystyle\sum_{x\in \f_q}\,f_j(x)= 1+ d_{q}\  \textnormal{if}  \ j=1; \\
& \displaystyle\sum_{x\in \f_q}\,f_j(x)=1+ d_{2q^2-4}-d_{2q^2-q-3}\  \textnormal{if}  \ j=2q^2-2;\\
& \displaystyle\sum_{x\in \f_q}\,f_j(x)=0\  \textnormal{otherwise}.
\end{split}
\]

\end{thm}

\begin{rmk}
\textnormal{When $q$ is even, it is easy to see that the sum $\sum_{a\in \f_q}f_{n}^2(a)$ agrees with Theorem~\ref{NT2.6}, where $d_n=\displaystyle\sum_{x\in \f_q}\,f_n^2(x)$.}
\end{rmk}

\noindent \textbf{Open Question:} Compute $\sum_{a\in \f_q}f_{n}^2(a)$ when $q$ is odd. 

\begin{rmk}
\textnormal{When $q$ is odd, computing the genrating function $\displaystyle\sum_{n=0}^{\infty}\,f_n(x)^2\,z^n$ turns out to be an arduous task since a useful identity for $f_{n-1}(x)f_{n-2}(x)$ is not known for Fibonacci polynomials.}
\end{rmk}

%%%%%%%%%%%%%%%%%%%%%%%%%%%%%%%%%%%%%%

\end{document}